\documentclass[reqno, 11pt]{article}

\usepackage{enumerate}
\usepackage{needspace}
\usepackage{latexsym}
\usepackage{amsmath}
\usepackage{amsfonts}
\usepackage{amsthm}
\usepackage{amssymb,mathtools}
\usepackage{newlfont}
\usepackage{graphics}
\usepackage{color}
\usepackage{float}
\usepackage{diagbox}
\usepackage{tocloft}
\usepackage{titlesec}
\usepackage{booktabs,array,longtable,enumitem,microtype}
\usepackage{extpfeil}
\usepackage{centernot}

\usepackage[linesnumbered,ruled,vlined]{algorithm2e}
\usepackage{url}

\usepackage[margin=1in]{geometry}

\usepackage[T1]{fontenc}
\usepackage{lmodern}
\usepackage{microtype}
\usepackage{rotating}
\usepackage{multirow}
\usepackage{extarrows}
\usepackage[sort,compress,numbers]{natbib}
\usepackage[utf8]{inputenc}
\usepackage{xcolor}
\usepackage{listings}
\usepackage{aliascnt}
\usepackage[pagebackref,colorlinks=true,linkcolor=blue,citecolor=red,urlcolor=blue]{hyperref}
\usepackage[nameinlink,noabbrev,capitalize]{cleveref}
\numberwithin{equation}{section}

\setlist[enumerate]{label=\textup{(\roman*)},leftmargin=*}
\newtheorem{theorem}{Theorem}[section]
\newaliascnt{proposition}{theorem}
\newtheorem{proposition}[proposition]{Proposition}
\aliascntresetthe{proposition}
\newaliascnt{lemma}{theorem}
\newtheorem{lemma}[lemma]{Lemma}
\aliascntresetthe{lemma}
\newaliascnt{corollary}{theorem}

\aliascntresetthe{corollary}

\theoremstyle{definition}
\newaliascnt{definition}{theorem}
\newtheorem{definition}[definition]{Definition}
\aliascntresetthe{definition}
\newaliascnt{remark}{theorem}
\newtheorem{remark}[remark]{Remark}
\aliascntresetthe{remark}
\newaliascnt{example}{theorem}

\aliascntresetthe{example}
\newaliascnt{conjecture}{theorem}

\aliascntresetthe{conjecture}
\newaliascnt{question}{theorem}
\newtheorem{question}[question]{Question}
\aliascntresetthe{question}
\newaliascnt{problem}{theorem}

\aliascntresetthe{problem}
\newaliascnt{alg}{theorem}

\aliascntresetthe{alg}

\allowdisplaybreaks[4]

\SetKwInput{KwInput}{Input}                
\SetKwInput{KwOutput}{Output}              

\newcommand{\Z}{\mathbb Z}
\newcommand{\R}{\mathbb R}

\DeclareMathOperator{\aff}{aff}
\DeclareMathOperator{\conv}{conv}
\DeclareMathOperator{\spanR}{span_{\mathbb R}}
\DeclareMathOperator{\Vol}{Vol}
\newcommand{\coeff}[2]{[#1^{#2}]}

\title{Ehrhart positivity of smooth polytopes with few facets}

\author{Feihu Liu$^{\color{blue} \dag}$, Jinlong Tang$^{\color{blue} \ddag}$, Sihao Tao$^{\color{blue} \P}$, and Zihao Zhang$^{\color{blue} \S}$
\\[2mm]
{\small $^{\color{blue} \dag}$ Center for Combinatorics, LPMC,}\\[-0.8ex]
{\small Nankai University, Tianjin 300071, P.R.~China}\\
{\small $^{\color{blue} \ddag, \P}$ School of Mathematical Sciences,}\\[-0.8ex]
{\small Capital Normal University, Beijing, 100048, P.R.~China}\\
{\small $^{\color{blue} \S}$ School of Mathematics and Statistics,}\\[-0.8ex]
{\small Beijing Institute of Technology, Beijing 102400, P.R.~China}\\
{\small {\color{blue} $^\dag$} Email address: \url{liufeihu7476@163.com}}\\
{\small {\color{blue} $^\ddag$} Email address: \url{jinlong\_tang@cnu.edu.cn}}\\
{\small {\color{blue} $^\P$} Email address: \url{sihao\_tao@cnu.edu.cn}}\\
{\small {\color{blue} $^\S$} Email address: \url{zihao-zhang@foxmail.com}}\\
}

\date{\today}

\begin{document}

\maketitle

\begin{abstract}
For each dimension $d$, we determine the minimum number of facets of a smooth $d$-dimensional lattice polytope whose Ehrhart polynomial has at least one negative coefficient. This minimum is six in dimension $d=3$, seven in dimension $d=4$, and $d+2$ in every dimension $d\ge 5$. The lower bounds follow from the classification of smooth lattice polytopes with at most $d+2$ facets, together with explicit formulas for Ehrhart polynomials in dimensions at most $4$. We also prove that a lattice polytope of dimension at least $3$ with a smooth vertex cannot be $2$-integral with respect to any ordered lattice basis.
\end{abstract}

\noindent
\begin{small}
\end{small}

\noindent
\begin{small}
\end{small}


\section{Introduction}

A \emph{lattice polytope} is the convex hull of finitely many points of $\Z^d$. Throughout this paper, a $d$-dimensional lattice polytope is assumed to be full-dimensional in $\R^d$. Its \emph{lattice-point counting function} is
\[
 L_P(t)=|tP\cap\Z^d|=\sum_{j=0}^{d}e_j(P)t^j,\qquad t\in \mathbb{N}.
\]
Ehrhart~\cite{Ehrhart,BR2015} proved that the function $L_P(t)$ is a polynomial in $t$ of degree $d$.
This polynomial is called the \emph{Ehrhart polynomial} of $P$.
The leading coefficient is the volume of $P$, the coefficient of $t^{d-1}$ is half its lattice boundary volume, and the constant coefficient is one. Our volume convention assigns volume one to a fundamental parallelepiped of the relevant lattice. In particular, the volume of a unimodular $j$-simplex is $1/j!$.

A polytope is \emph{Ehrhart positive} if $e_j(P)>0$ for every $0\le j\le d$.
Ehrhart positivity has attracted considerable attention in recent years. 
We refer the reader to \cite{Liu2019} for a survey of this topic.

A nonzero vector of $\Z^d$ is \emph{primitive} if it is not an integer multiple $m u$ of another lattice vector with $m\ge2$. A vertex $v$ of a lattice $d$-polytope is \emph{smooth} if exactly $d$ edges meet at $v$ and their primitive directions form a basis of $\Z^d$. A polytope is \emph{smooth} if all its vertices are smooth. This is also called the \emph{Delzant condition}. Smooth polytopes occur as moment polytopes of nonsingular projective toric varieties; background on this correspondence is available in \cite{CLS2011}.

Castillo, Liu, Nill, and Paffenholz \cite[Theorem~1.2]{CLNP2018} proved that in every dimension $d\ge3$ a smooth lattice polytope can have negative coefficients in all degrees $1,\ldots,d-2$. Their examples arise from repeated vertex truncations followed by products. In \cite[Example~1.5]{CLNP2018}, this given a smooth reflexive polytope of dimension $9$ with $12$ facets and linear Ehrhart coefficient $-6673/630$. Thus the assertion that every smooth $d$-polytope with at most $d+3$ facets is Ehrhart positive already fails. 

This example partially answers the following open problem.
\begin{question}[\cite{AimPl-Website}]\label{Question111}
Are $d$-dimensional smooth polytopes with $\leq d+3$ facets Ehrhart positive or satisfy the higher integrality condition?
\end{question}

The restriction on the number of facets nevertheless suggests a more precise question: 

\begin{question}
How few facets can a smooth polytope with a negative Ehrhart coefficient have?
\end{question}

For $d\ge1$, define
\[
 m_d=\min\{f_{d-1}(P):P\text{ is a smooth lattice }d\text{-polytope and } e_j(P)<0\text{ for some }1\le j\le d\},
\]
where $f_{d-1}(P)$ denotes the number of facets and the minimum of the empty set is $+\infty$.

\begin{theorem}\label{thm:minimum}
The minimum facet numbers are
\[ m_d= \begin{cases}
 +\infty,&d=1,2,\\
 d+3,&d=3,4,\\
 d+2,&d\ge5.
 \end{cases}
\]
More precisely, every smooth $d$-polytope with $d+1$ facets is Ehrhart positive, and every smooth $d$-polytope with $d+2$ facets is Ehrhart positive when $d\le4$. Counterexamples with exactly $d+2$ facets exist in every dimension $d\ge5$.
\end{theorem}

The lower bounds in \cref{thm:minimum} use the classification of smooth polytopes with few facets. The case of $d+2$ facets is due to Kleinschmidt \cite{Kleinschmidt1988}; the corresponding polytopes are described explicitly in \cite[\S4.1]{HNS2022}. The proof of positivity in dimensions at most four reduces to the first two moments of weak compositions and the power sums of degrees at most three. The upper bounds follow from two families whose slices are simplices or products of a simplex and an interval. 

\cref{Question111} concerns higher integrality in the sense of Liu \cite{Liu2011}. After an ordered basis of $\Z^d$ has been chosen, let $\pi_j:\R^d\longrightarrow\R^j$ be projection onto the first $j$ coordinates. A lattice polytope is \emph{$k$-integral} if
\[
 \pi_j\bigl(\aff(F)\cap\Z^d\bigr)=\Z^j
\]
for every $j$-dimensional face $F$ and every $0\le j\le k$. For $d\ge2$, Liu's coefficient theorem implies that $(d-2)$-integrality is sufficient for Ehrhart positivity. This condition concerns a common ordered lattice basis for all faces; ordinary lattice integrality alone does not imply it.
We also obtain the following result.
\begin{theorem}\label{thm:obstruction}
A lattice polytope of dimension $d\ge3$ with a smooth vertex is not $2$-integral with respect to any ordered lattice basis. Consequently, a smooth $d$-polytope with $d\ge4$ is not $(d-2)$-integral with respect to any ordered lattice basis, regardless of its number of facets.
\end{theorem}

The paper is organized as follows. 
\cref{sec:preliminaries} records the enumerative results used in the proofs. 
\cref{sec:positive} proves positivity for smooth polytopes with at most $d+2$ facets in dimensions at most $4$. \cref{sec:opposing,sec:ranktwo-negative} construct the counterexamples and complete the proof of \cref{thm:minimum}. \cref{sec:integrality} proves \cref{thm:obstruction} and discusses its relation to Liu's condition.

\section{Preliminaries}\label{sec:preliminaries}

An \emph{affine unimodular transformation} of $\R^d$ is a map $x\longmapsto Ux+a$, where $U$ is an integer matrix with determinant $1$ or $-1$ and $a\in\Z^d$. Such a map preserves smoothness, facet numbers, and Ehrhart polynomials. Indeed, for every nonnegative integer $t$, the map $x\longmapsto Ux+ta$ is a bijection between the lattice points of $tP$ and those of $t(UP+a)$.

The standard lattice $p$-simplex is
\[
 \Delta_p=\{x\in\R^p:x_i\ge0\ (1\le i\le p),\ x_1+\cdots+x_p\le1\}.
\]
For every nonnegative integer $N$, its dilation contains
\[
 |N\Delta_p\cap\Z^p|=\binom{N+p}{p}
\]
lattice points. This is the standard weak-composition formula: the slack coordinate $N-x_1-\cdots-x_p$ gives a bijection with weak compositions of $N$ into $p+1$ parts; see \cite[Chapter~1]{Stanley2012}.

To check smoothness at a vertex, we shall list all edges through that vertex and calculate their primitive directions. If exactly $d$ edges meet there, arrange these directions as the columns of an integer $d\times d$ matrix. They form a basis of $\Z^d$ precisely when the determinant is $1$ or $-1$. Indeed, an integer matrix with either determinant has an integer inverse by the cofactor formula; conversely, a matrix and its integer inverse have integer determinants whose product is one. All coordinates below refer to the displayed standard lattices.

\subsection{Smooth polytopes with \texorpdfstring{$d+1$ or $d+2$}{d+1 or d+2} facets}

The classification of smooth polytopes with $d+2$ facets is due to Kleinschmidt \cite{Kleinschmidt1988}. Its explicit description as a convex hull of simplices is given in \cite[\S4.1]{HNS2022}. In the coordinates below, the last $q$ coordinates range over a dilated simplex, and fixing those coordinates leaves a dilated simplex in the first $p$ coordinates. This description will allow us to count lattice points by the weak-composition formula.

\begin{theorem}\label{thm:classification}
Let $P$ be a smooth lattice $d$-polytope.
\begin{enumerate}
\item If $P$ has $d+1$ facets, then $P$ is affinely unimodularly equivalent to $a\Delta_d$ for a positive integer $a$.
\item If $P$ has $d+2$ facets, then there are positive integers $p,q,a,b$ with $p+q=d$ and nonnegative integers $c_1,\ldots,c_q$ such that $P$ is affinely unimodularly equivalent to
\begin{equation}\label{eq:normalform}
 \left\{(x,y)\in\R^p\times\R^q:
 \begin{array}{l}
 y_j\ge0\ (1\le j\le q),\quad y_1+\cdots+y_q\le b,\\
 x_i\ge0\ (1\le i\le p),\quad
 x_1+\cdots+x_p\le a+c_1y_1+\cdots+c_qy_q
 \end{array}\right\}.
\end{equation}
\end{enumerate}
\end{theorem}

\begin{proof}[Proof of (1)]
A $d$-polytope with $d+1$ facets is a simplex. Translate one vertex to the origin. The primitive directions of the $d$ edges from this vertex form a lattice basis, so an integer change of coordinates with determinant $1$ or $-1$ sends them to the standard basis vectors $e_1,\ldots,e_d$. In these coordinates,
\[
 P=\conv(0,a_1e_1,\ldots,a_de_d)
\]
for positive integers $a_1,\ldots,a_d$. If $d=1$, this already gives the required form. Suppose that $d\ge2$, and fix $i\in\{1,\ldots,d\}$. The primitive edge directions at $a_i e_i$ are
\[-e_i\qquad \text{and}\qquad\frac{a_je_j-a_ie_i}{\gcd(a_i,a_j)}\quad (1\le j\le d,\ j\ne i).
\]
The absolute value of their determinant is
\[
 \prod_{\substack{1\le j\le d\\ j\ne i}}
 \frac{a_j}{\gcd(a_i,a_j)}.
\]
Smoothness at $a_i e_i$ implies that this product equals one. Each factor is a positive integer, so $\gcd(a_i,a_j)=a_j$ for every $j\ne i$. Thus $a_j$ divides $a_i$. Since $i$ was arbitrary, any two of the $a_i$ divide one another, and hence $a_1=\cdots=a_d$. Their common value is the required integer $a$.
\end{proof}

To express the classification in (2) directly in terms of simplices, let $e_1,\ldots,e_q$ now denote the standard basis of $\R^q$. The convex-hull description in \cite[\S4.1]{HNS2022}, with its integral parameters, is
\[
 \conv\left(
 (a\Delta_p\times\{0\})\ \cup\
 \bigcup_{j=1}^{q}
 \bigl((a+bc_j)\Delta_p\times\{be_j\}\bigr)
 \right).
\]
The parameters $a$ and $b$ are positive integers, and the $c_j$ are nonnegative integers. The simplex in the first $p$ coordinates has dilation factor $a$ at $y=0$ and $a+bc_j$ at $y=be_j$. Thus $c_j$ records the increase in the dilation factor per unit change of the coordinate $y_j$.

Here is the equivalence between this convex hull and \eqref{eq:normalform}. Every $y\in b\Delta_q$ has the unique expression
\[
 y=\sum_{j=1}^{q}\theta_jbe_j,
 \qquad
 \theta_j=\frac{y_j}{b}\ (1\le j\le q),
 \qquad
 \theta_0=1-\sum_{j=1}^{q}\theta_j,
\]
where $\theta_0,\ldots,\theta_q\ge0$ and $\sum_{j=0}^{q}\theta_j=1$. If $(x,y)$ belongs to the displayed convex hull, then
\[
 x=\theta_0\xi_0+\sum_{j=1}^{q}\theta_j\xi_j,
 \qquad
 \xi_0\in a\Delta_p,
 \qquad
 \xi_j\in(a+bc_j)\Delta_p\qquad (1\le j\le q).
\]
Consequently, $x_i\ge0$ for every $1\le i\le p$, and
\[
 \sum_{i=1}^{p}x_i
 \le\theta_0a+\sum_{j=1}^{q}\theta_j(a+bc_j)
 =a+\sum_{j=1}^{q}c_jy_j.
\]
Conversely, suppose that $(x,y)$ satisfies \eqref{eq:normalform}. Since $a+\sum_{j=1}^{q}c_jy_j>0$, the vector
\[
 u=\frac{x}{a+\sum_{j=1}^{q}c_jy_j}
\]
belongs to $\Delta_p$. The identity
\[
 (x,y)=\theta_0(au,0)
 +\sum_{j=1}^{q}\theta_j\bigl((a+bc_j)u,be_j\bigr)
\]
places $(x,y)$ in the convex hull. This proves the claimed equivalence. In particular, at each fixed $y\in b\Delta_q$, the allowed first coordinates form exactly
\[
 \left(a+\sum_{j=1}^{q}c_jy_j\right)\Delta_p.
\]
If $c_1=\cdots=c_q=0$, the polytope is the product $a\Delta_p\times b\Delta_q$.

\subsection{Power sums}

For nonnegative integers $m,n$, define
\[
 S_m(n)=\sum_{j=0}^{n}j^m,
\]
where $j^0=1$ also for $j=0$. In particular, $S_0(n)=n+1$, $S_1(n)=\frac{n(n+1)}2$ and 
\begin{equation}\label{eq:small-powers}
 S_2(n)=\frac{n(n+1)(2n+1)}6,\ 
 S_3(n)=\frac{n^2(n+1)^2}4,\ 
 S_4(n)=\frac{n(n+1)(2n+1)(3n^2+3n-1)}{30}.
\end{equation}
We use these classical identities, and the general power-sum formula, as established results; see \cite[\S24.4(iii)]{DLMF}. To fix the sign convention, the Bernoulli numbers are defined by
\[
 \frac{z}{e^z-1}=\sum_{k=0}^{\infty}B_k\frac{z^k}{k!}.
\]
Thus $B_0=1$, $B_1=-1/2$, $B_2=1/6$, $B_3=0$, and $B_4=-1/30$. For $m\ge1$, the required version of the power-sum formula is
\begin{equation}\label{eq:faulhaber}
 S_m(t)=\frac1{m+1}\sum_{k=0}^{m}(-1)^k
 \binom{m+1}{k}B_k t^{m+1-k}.
\end{equation}
Its right-hand side is the polynomial agreeing with $S_m(t)$ for nonnegative integer $t$. In particular, for $m\ge4$,
\begin{equation}\label{eq:bernoulli-coefficient}
 \coeff{t}{m-3}S_m(t)
 =\frac{\binom{m+1}{4}B_4}{m+1}
 =-\frac{m(m-1)(m-2)}{720}.
\end{equation}
The upper limit $m$ in \eqref{eq:faulhaber} excludes the constant term, which is zero when $m\ge1$. The separate formula $S_0(t)=t+1$ will be used whenever $m=0$.

\section{Positivity with at most two additional facets}\label{sec:positive}

\begin{proposition}\label{prop:positive}
Every smooth lattice $d$-polytope with $d+1$ facets is Ehrhart positive. Every smooth lattice $d$-polytope with $d+2$ facets is Ehrhart positive if $d\le4$.
\end{proposition}

\begin{proof}
The first assertion follows from \cref{thm:classification}(i). Indeed,
\[
 L_{a\Delta_d}(t)=\binom{at+d}{d}
 =\frac1{d!}\prod_{i=1}^{d}(at+i).
\]
Each factor has positive coefficients, and the product has a positive coefficient in every degree from zero to $d$.

For the second assertion, \cref{thm:classification}(ii) allows us to assume that $P$ has the form \eqref{eq:normalform}. For a fixed lattice point $y\in bt\Delta_q$, the allowed $x$-coordinates form the simplex $(at+c_1y_1+\cdots+c_qy_q)\Delta_p$. The number of its lattice points is
\[
 \binom{at+c_1y_1+\cdots+c_qy_q+p}{p}.
\]
All arguments of this binomial coefficient are nonnegative integers because $a>0$ and $c_j\ge0$. Consequently,
\begin{equation}\label{eq:normalform-count}
 L_P(t)=
 \sum_{y_1=0}^{bt}\sum_{y_2=0}^{bt-y_1}\cdots
 \sum_{y_q=0}^{bt-y_1-\cdots-y_{q-1}}
 \binom{at+c_1y_1+\cdots+c_qy_q+p}{p}.
\end{equation}
When $q=1$, this notation means the single sum with $0\le y_1\le bt$.

\emph{The case $p=1$.}
The number of weak compositions of $bt$ into $q+1$ parts is $\binom{bt+q}{q}$. Their coordinate sums are equal by permutation symmetry. Since the sum of all $q+1$ coordinates in each composition is $bt$, the sum of any specified coordinate over all compositions is
\[
 \frac{bt}{q+1}\binom{bt+q}{q}.
\]
The summand in \eqref{eq:normalform-count} is now $at+c_1y_1+\cdots+c_qy_q+1$. Summing its constant and coordinate terms separately gives
\[
 L_P(t)=\binom{bt+q}{q}
 \left(1+\left(a+\frac{b(c_1+\cdots+c_q)}{q+1}\right)t\right).
\]
The binomial factor has positive coefficients in degrees zero through $q$. The second factor has positive constant and linear coefficients. Their product is therefore Ehrhart positive.

\emph{The case $q=1$ and $1\le p\le3$.}
Equation~\eqref{eq:normalform-count} becomes
\[
 L_P(t)=\frac1{p!}\sum_{j=0}^{bt}\prod_{i=1}^{p}(at+c_1j+i).
\]
The displayed product, before summation, is a polynomial in $t$ and $j$ with nonnegative coefficients. Its degree in $j$ is at most $p$. Every monomial $t^u j^v$, where $0\le v\le p\le3$, contributes $t^uS_v(bt)$ after summation. The formulas in \eqref{eq:small-powers} show that each of these contributions has nonnegative coefficients as a polynomial in $t$.

It remains to verify positivity in every degree. The part of the product independent of $j$ contributes
\[
 \frac{bt+1}{p!}\prod_{i=1}^{p}(at+i),
\]
which has positive coefficients in all degrees from zero to $p+1=d$. All the remaining contributions have nonnegative coefficients. Thus every coefficient of $L_P(t)$ is positive.

\emph{The case $p=q=2$.}
This is the only case with $p+q\le4$ not covered above. The pairs $(y_1,y_2)$ with $y_1,y_2\ge0$ and $y_1+y_2\le bt$ correspond to weak compositions $(y_0,y_1,y_2)$ of $bt$ into three parts, where $y_0=bt-y_1-y_2$. Their number is $\binom{bt+2}{2}$.
We have 
\[\sum_{y_1=0}^{bt}\sum_{y_2=0}^{bt-y_1}y_1=\sum_{y_1=0}^{bt}\sum_{y_2=0}^{bt-y_1}y_2=\frac{bt}{3}\binom{bt+2}{2}
\]
and
\begin{equation}\label{eq:composition-moments}
\sum_{y_1=0}^{bt}\sum_{y_2=0}^{bt-y_1}y_1^2=\frac{bt(bt+1)}6\binom{bt+2}{2},\qquad 
 \sum_{y_1=0}^{bt}\sum_{y_2=0}^{bt-y_1}y_1y_2=\frac{bt(bt-1)}{12}\binom{bt+2}{2}.
\end{equation}
The summand in \eqref{eq:normalform-count} is
\[
 \binom{at+c_1y_1+c_2y_2+2}{2}
 =\frac12\left((at+1)(at+2)+(2at+3)(c_1y_1+c_2y_2)+(c_1y_1+c_2y_2)^2\right).
\]
By \eqref{eq:composition-moments}, we obtain
\begin{small}
\begin{align*}
 L_P(t)=\frac12\binom{bt+2}{2}\bigg((at+1)(at+2)+\frac{(2at+3)bt(c_1+c_2)}3+\frac{(bt)^2(c_1^2+c_1c_2+c_2^2)
+n(c_1^2-c_1c_2+c_2^2)}6\bigg).
\end{align*}
\end{small}
The coefficient $c_1^2-c_1c_2+c_2^2=(c_1-c_2)^2+c_1c_2$ is nonnegative. All coefficients of the expression in parentheses, as a polynomial in $t$, are consequently nonnegative. Its term $(at+1)(at+2)$, together with the prefactor, contributes a positive coefficient in every degree from zero to four. Thus $P$ is Ehrhart positive. The three cases exhaust all positive integers $p,q$ with $p+q\le4$.
\end{proof}

\section{A construction with three additional facets}\label{sec:opposing}

For integers $r\ge1$ and $c\ge0$, consider
\begin{equation}\label{eq:opposing-family}
 R_{r,c}=\left\{(x,y,z)\in\R^r\times\R\times\R:
 \begin{array}{l}
 0\le z\le1,\quad x_i\ge0\ (1\le i\le r),\\
 x_1+\cdots+x_r\le1+cz,\quad 0\le y\le c+1-cz
 \end{array}\right\}.
\end{equation}
For each fixed $z\in[0,1]$, the allowed $(x,y)$-coordinates form the product
\[
 (1+cz)\Delta_r\times[0,c+1-cz].
\]
As $z$ increases, the dilation factor of the simplex increases and the length of the interval decreases. Both are at least one throughout $[0,1]$, so neither factor loses dimension, including at $z=0$ and $z=1$.

\begin{proposition}\label{prop:opposing-geometry}
The polytope $R_{r,c}$ has dimension $r+2$, exactly $r+5$ facets, and $4(r+1)$ vertices. It is a smooth lattice polytope. Its Ehrhart polynomial satisfies
\begin{equation}\label{eq:opposing-count}
 L_{R_{r,c}}(t)=\sum_{j=0}^{t}
 \binom{t+cj+r}{r}\bigl((c+1)t-cj+1\bigr).
\end{equation}
\end{proposition}

\begin{proof}
For each fixed $z\in[0,1]$, the allowed pairs $(x,y)$ form
\[
 (1+cz)\Delta_r\times[0,c+1-cz].
\]
Both $1+cz$ and $c+1-cz$ are at least one. A point with
$0<z<1$, all $x_i>0$, $x_1+\cdots+x_r<1+cz$, and
$0<y<c+1-cz$ satisfies every defining inequality strictly.
Such points exist, so $R_{r,c}$ has dimension $r+2$.

Each inequality in \eqref{eq:opposing-family} defines a facet.
The faces $z=0$ and $z=1$ are products of an $r$-simplex and
a nondegenerate interval, so each has dimension $r+1$.
On the face $y=0$, choose $0<z<1$ and choose all $x_i>0$
with $x_1+\cdots+x_r<1+cz$. Every other defining inequality
is then strict, so this face has dimension $r+1$.
The same argument applies to $y=c+1-cz$.
On $x_i=0$, choose the remaining $x$-coordinates positive
with sum less than $1+cz$, and choose $z$ and $y$ in the
interiors of their allowed intervals. This again gives a
relatively open set of dimension $r+1$; when $r=1$, there
are no remaining $x$-coordinates to choose.
Finally, the equation $x_1+\cdots+x_r=1+cz$ allows all
$x_i$ to be positive, with $0<z<1$ and $0<y<c+1-cz$,
and also defines a face of dimension $r+1$.
The $r+5$ supporting hyperplanes are distinct. Thus there
are exactly $r+5$ facets.

Every vertex has $z=0$ or $z=1$. To see this, consider a
point $(x,y,z)\in R_{r,c}$ with $0<z<1$, and set
\[
 u=\frac{x}{1+cz}\in\Delta_r,
 \qquad s=\frac{y}{c+1-cz}\in[0,1].
\]
Then
\[
 (x,y,z)
 =(1-z)\bigl(u,(c+1)s,0\bigr)
   +z\bigl((c+1)u,s,1\bigr).
\]
The two points on the right belong to $R_{r,c}$ and have
different last coordinates. The point $(x,y,z)$ therefore
lies in the interior of a nondegenerate segment contained
in $R_{r,c}$ and cannot be a vertex.
Let $e_1,\ldots,e_r$ be the standard basis of $\R^r$, and
set $e_0=0$. The two endpoint faces are products of a simplex
and an interval. Their vertices, and hence all the vertices
of $R_{r,c}$, are
\[
 v_{i,\delta,\eta}
 =\bigl((1+c\eta)e_i,\ \delta(c+1-c\eta),\ \eta\bigr),
 \qquad 0\le i\le r,\quad \delta,\eta\in\{0,1\}.
\]
Every vertex of an endpoint face is a vertex of the whole
polytope. These $4(r+1)$ points are distinct lattice points.

Fix a vertex $v_{i,\delta,\eta}$. Exactly $r+2$ facet
equations hold at this vertex: $r$ equations from the
simplex inequalities, one equation for $y$, and $z=\eta$.
These equations are independent. Indeed, after $z$ is fixed,
the simplex equations determine $x$ uniquely as
$(1+c\eta)e_i$, and the equation for $y$ determines
$y=\delta(c+1-c\eta)$.
The edges incident to this vertex are precisely
\[
 [v_{i,\delta,\eta},v_{k,\delta,\eta}]
 \quad(0\le k\le r,\ k\ne i),\qquad
 [v_{i,\delta,\eta},v_{i,1-\delta,\eta}],
 \qquad
 [v_{i,\delta,\eta},v_{i,\delta,1-\eta}].
\]
The first $r$ segments are the simplex edges in the face
$z=\eta$; the next is the interval edge in that face.
For the last segment, retain the $r$ simplex equations
and the equation for $y$, and allow $z$ to vary. These
equations give
\[
 x=(1+cz)e_i,\qquad y=\delta(c+1-cz),\qquad 0\le z\le1,
\]
which describes exactly the segment joining
$v_{i,\delta,0}$ to $v_{i,\delta,1}$.
Thus each listed segment is a one-dimensional intersection
of facets. They exhaust the incident edges: at a vertex
where exactly $r+2$ independent facet equations hold,
each incident edge is obtained by retaining $r+1$ of
those equations, and the list above accounts for each
possible omitted equation.

The primitive directions of these edges, pointing away
from $v_{i,\delta,\eta}$, are
\[
 (e_k-e_i,0,0)
 \quad(0\le k\le r,\ k\ne i),\qquad
 (0,1-2\delta,0),
 \qquad
 (1-2\eta)(c e_i,-c\delta,1).
\]
For the simplex edges, the difference of the endpoints is
$(1+c\eta)(e_k-e_i,0,0)$; the vector $e_k-e_i$ is primitive
because one of its coordinates equals $1$ or $-1$.
For the interval edge, the difference is
$(c+1-c\eta)(0,1-2\delta,0)$.
For the edge joining the two endpoint faces, the displayed
direction is already the difference of the endpoints and
is primitive because its last coordinate is $1$ or $-1$.

The $r$ vectors $e_k-e_i$, with $k\ne i$, form a basis
of $\Z^r$. For $i=0$, they are the standard basis.
For $i\ge1$, they consist of $-e_i$ and $e_k-e_i$ for
$1\le k\le r$, $k\ne i$. Subtracting $-e_i$ from each
of the latter vectors gives $e_k$, so these vectors also
form an integer basis. Let $B_i$ be the matrix with these
$r$ vectors as columns in any fixed order, and set
$\sigma=1-2\delta$ and $\tau=1-2\eta$.
The matrix whose columns are the primitive directions
listed above is
\[
 \begin{pmatrix}
  B_i&0&\tau c e_i\\
  0&\sigma&-\tau c\delta\\
  0&0&\tau
 \end{pmatrix}.
\]
Expansion along the last row, followed by expansion along
the row for the $y$-coordinate, gives determinant
$\tau\sigma\det B_i=1$ or $-1$.
Consequently, the $r+2$ primitive incident edge directions
form a basis of $\Z^{r+2}$. This proves smoothness at every
vertex.

For a nonnegative integer $t$, the inequalities of $tR_{r,c}$ are
\[
 0\le z\le t,\qquad x_i\ge0\ (1\le i\le r),\qquad
 x_1+\cdots+x_r\le t+cz,\qquad
 0\le y\le(c+1)t-cz.
\]
At an integer height $z=j$, there are
$\binom{t+cj+r}{r}$ choices for the integer vector $x$
and $(c+1)t-cj+1$ choices for the integer $y$.
These choices are independent. Summation over
$j=0,\ldots,t$ gives \eqref{eq:opposing-count} for every
nonnegative integer $t$, and hence determines the Ehrhart
polynomial.
\end{proof}

\begin{proposition}\label{prop:three-dimensional}
For every integer $c\ge0$,
\begin{equation}\label{eq:three-dimensional-polynomial}
 L_{R_{1,c}}(t)=
 \left(\frac{c^2}{6}+c+1\right)t^3
 +(2c+3)t^2
 +\left(-\frac{c^2}{6}+c+3\right)t+1.
\end{equation}
Its linear coefficient is negative precisely when $c\ge9$.
\end{proposition}

\begin{proof}
The summand in \eqref{eq:opposing-count} for $r=1$ has the identity
\[
 (t+cj+1)\bigl((c+1)t-cj+1\bigr)
 =(t+1)\bigl((c+1)t+1\bigr)+c^2j(t-j).
\]
The first term is independent of $j$, while the second involves only the first two power sums. Thus
\[ L_{R_{1,c}}(t)=(t+1)^2\bigl((c+1)t+1\bigr)+c^2\bigl(tS_1(t)-S_2(t)\bigr)
 =(t+1)^2\bigl((c+1)t+1\bigr)+\frac{c^2}{6}(t^3-t).
\]
Expansion gives \eqref{eq:three-dimensional-polynomial}. The condition on the linear coefficient is
\[
 -\frac{c^2}{6}+c+3<0
 \quad\Longleftrightarrow\quad
 c^2-6c-18>0.
\]
The positive root of the quadratic is $3+3\sqrt3$, which lies between eight and nine. Since $c$ is a nonnegative integer, the coefficient is negative exactly when $c\ge9$.
\end{proof}

For $c=9$, \cref{prop:opposing-geometry} gives a smooth three-dimensional lattice polytope with six facets and eight vertices, and \eqref{eq:three-dimensional-polynomial} gives
\[
 L_{R_{1,9}}(t)=\frac{47}{2}t^3+21t^2-\frac32t+1.
\]
For example, the two integer slices at dilation one each contain $22$ lattice points, so $L_{R_{1,9}}(1)=44$, in agreement with the polynomial.

\begin{remark}
The polytope $R_{1,c}$ is the convex hull of
\[
 [0,1]\times[0,c+1]\times\{0\}
 \quad\text{and}\quad
 [0,c+1]\times[0,1]\times\{1\}.
\]
This construction is called the Cayley sum of the two rectangles. Smooth Cayley sums of boxes and their lattice-point counts are treated in \cite[Proposition~2.3 and Theorem~3.1]{Ferroni2026}. The computation above specializes this construction and identifies the sign of each Ehrhart coefficient.
\end{remark}

For $r=2$, the simplex factor in \eqref{eq:opposing-count} is quadratic. A direct expansion gives
\begin{small}
\begin{equation}\label{eq:four-dimensional-expansion}
2L_{R_{2,c}}(t)=(t^2+3t+2)\bigl((c+1)t+1\bigr)S_0(t)+c\bigl((2c+1)t^2+(3c+2)t+1\bigr)S_1(t)+c^2\bigl((c-1)t-2\bigr)S_2(t)-c^3S_3(t).
\end{equation}
\end{small}

At $c=12$, the formulas in \eqref{eq:small-powers} turn \eqref{eq:four-dimensional-expansion} into
\begin{equation}\label{eq:four-dimensional-example}
2L_{R_{2,12}}(t)=259t^4+263t^3-9t^2-11t+2.
\end{equation}
The polytope is smooth and has seven facets by \cref{prop:opposing-geometry}. 
Both coefficients that can be negative in dimension four are negative in this example.

The family $R_{r,c}$ also gives counterexamples with exactly $d+3$ facets in every dimension $d\ge4$, with a negative coefficient of $t^{d-2}$.

\begin{proposition}
For every integer $r\ge2$, we have
\[\coeff{t}{r}L_{R_{r,c}}(t)=-\frac{c^{r+1}}{12r!}+O(c^r)\qquad(c\longrightarrow+\infty).
\]
Thus, for each $d\ge4$, the coefficient of $t^{d-2}$ is negative in a smooth $d$-polytope with exactly $d+3$ facets when $c$ is sufficiently large.
\end{proposition}

\begin{proof}
The binomial polynomial in \eqref{eq:opposing-count} equals
\[
 \binom{t+cj+r}{r}=\frac1{r!}\prod_{i=1}^{r}(t+cj+i).
\]
Its coefficient of $c^r$ is $j^r/r!$. The other factor in the summand is $(t+1)+c(t-j)$, whose coefficient of $c$ is $t-j$. Therefore the coefficient of $c^{r+1}$ in the whole lattice-point count is
\[
 \frac1{r!}\sum_{j=0}^{t}j^r(t-j)
 =\frac1{r!}\bigl(tS_r(t)-S_{r+1}(t)\bigr).
\]
Every remaining power of $c$ has degree at most $r$; its coefficient is a polynomial in $t$ because it is a finite linear combination of power sums. Thus the coefficient of $t^r$ in the Ehrhart polynomial is itself a polynomial in $c$ of degree at most $r+1$.

By Equation~\eqref{eq:faulhaber} with $B_2=1/6$, we obtain
\[\coeff{t}{r-1}S_r(t)=\frac r{12},
 \qquad
 \coeff{t}{r}S_{r+1}(t)=\frac{r+1}{12}
 \qquad(r\ge2).
\]
The coefficient of $t^r$ in $tS_r(t)-S_{r+1}(t)$ is consequently $-1/12$. This proves the displayed asymptotic formula. Dividing it by $c^{r+1}$ gives a limit equal to $-1/(12r!)$, which is negative. The coefficient is therefore negative for every sufficiently large integer $c$. Finally, $d=r+2$ and $r+5=d+3$ by \cref{prop:opposing-geometry}.
\end{proof}

\section{Two additional facets in dimensions at least five}\label{sec:ranktwo-negative}

A related family is obtained by retaining only the simplex coordinates and the height coordinate. For integers $r\ge1$ and $c\ge0$, define
\begin{equation}\label{eq:ranktwo-family}
 Q_{r,c}=\{(x,z)\in\R^r\times\R:
 0\le z\le1,\ x_i\ge0\ (1\le i\le r),\
 x_1+\cdots+x_r\le1+cz\}.
\end{equation}

\begin{proposition}\label{prop:ranktwo-geometry}
The polytope $Q_{r,c}$ is a smooth lattice polytope of dimension $r+1$, with exactly $r+3$ facets and $2(r+1)$ vertices. Its Ehrhart polynomial is
\begin{equation}\label{eq:ranktwo-count}
 L_{Q_{r,c}}(t)=\sum_{j=0}^{t}\binom{t+cj+r}{r}.
\end{equation}
\end{proposition}

\begin{proof}
For each fixed $z\in[0,1]$, the inequalities in
\eqref{eq:ranktwo-family} allow exactly the $x$-coordinates
in $(1+cz)\Delta_r$. Its dilation factor is at least one. There are points with $0<z<1$, all $x_i>0$,
and $x_1+\cdots+x_r<1+cz$. Such points satisfy every
defining inequality strictly, so $Q_{r,c}$ has dimension
$r+1$.

The faces $z=0$ and $z=1$ are $r$-dimensional simplices.
On a face $x_i=0$, choose $0<z<1$ and choose the remaining
$x$-coordinates positive with sum less than $1+cz$.
This gives a relatively open subset of dimension $r$;
when $r=1$, the face is the segment $x_1=0$, $0\le z\le1$.
On the face $x_1+\cdots+x_r=1+cz$, choose all $x_i>0$
and $0<z<1$, again obtaining a relatively open subset of
dimension $r$. Thus each displayed inequality defines
a facet. Their supporting hyperplanes are distinct, so
there are exactly $r+3$ facets.

Suppose that $(x,z)\in Q_{r,c}$ and $0<z<1$.
Set $u=x/(1+cz)\in\Delta_r$. The identity
\[
 (x,z)=(1-z)(u,0)+z((c+1)u,1)
\]
expresses $(x,z)$ as a point in the interior of a segment
joining two distinct points of $Q_{r,c}$. It is therefore
not a vertex. Let $e_1,\ldots,e_r$ be the standard basis
of $\R^r$, and set $e_0=0$. The complete list of vertices is
\[
 w_{i,\eta}=\bigl((1+c\eta)e_i,\eta\bigr),
 \qquad 0\le i\le r,\quad \eta\in\{0,1\}.
\]
Indeed, these are precisely the vertices of the two
endpoint faces, and a vertex of either face is a vertex
of the whole polytope. The $2(r+1)$ listed points are
distinct lattice points.

At $w_{i,\eta}$, exactly $r+1$ facet equations hold:
the $r$ simplex equations determining $x=(1+c\eta)e_i$
and the equation $z=\eta$. They are independent.
The incident edges are
\[
 [w_{i,\eta},w_{k,\eta}]
 \quad(0\le k\le r,\ k\ne i),
 \qquad
 [w_{i,\eta},w_{i,1-\eta}].
\]
The first $r$ are the edges of the simplex face $z=\eta$.
For the last edge, retain the $r$ simplex equations and
allow $z$ to vary. Their solutions in $Q_{r,c}$ are exactly
\[
 ((1+cz)e_i,z),\qquad 0\le z\le1,
\]
so they form the segment joining $w_{i,0}$ to $w_{i,1}$.
These are all the incident edges, since each is obtained
by omitting one of the $r+1$ independent facet equations
at the vertex.
Their primitive directions, pointing away from
$w_{i,\eta}$, are
\[
 (e_k-e_i,0)
 \quad(0\le k\le r,\ k\ne i),
 \qquad
 (1-2\eta)(c e_i,1).
\]
The first $r$ vectors are primitive because each has a
coordinate equal to $1$ or $-1$. The difference of the
endpoints of the last edge is the final displayed vector;
its last coordinate is $1$ or $-1$, so it is primitive
for every $c\ge0$.

The vectors $e_k-e_i$, $k\ne i$, form a basis of $\Z^r$.
When $i=0$, they are the standard basis. When $i\ge1$,
they are $-e_i$ and $e_k-e_i$ for $1\le k\le r$, $k\ne i$;
subtracting $-e_i$ from each of the latter vectors produces
the remaining standard basis vectors.
Let $B_i$ be the matrix with these vectors as columns,
and set $\tau=1-2\eta$. The matrix of all the primitive
incident edge directions is
\[
 \begin{pmatrix}
  B_i&\tau c e_i\\
  0&\tau
 \end{pmatrix}.
\]
Its determinant, obtained by expansion along the last row,
is $\tau\det B_i=1$ or $-1$.
Thus the $r+1$ primitive incident edge directions form
a basis of $\Z^{r+1}$, proving smoothness at every vertex.

For a nonnegative integer $t$, an integer height $z=j$
in $tQ_{r,c}$ satisfies $0\le j\le t$. At this height,
the possible integer vectors $x$ satisfy
\[
 x_i\ge0\ (1\le i\le r),\qquad
 x_1+\cdots+x_r\le t+cj.
\]
The weak-composition formula gives
$\binom{t+cj+r}{r}$ such vectors. Summation over
$j=0,\ldots,t$ proves \eqref{eq:ranktwo-count}.
\end{proof}

\begin{theorem}\label{thm:ranktwo-negative}
For each integer $r\ge4$, we have 
\[\coeff{t}{r-3}L_{Q_{r,c}}(t)=-\frac{c^r}{720(r-3)!}+O(c^{r-1})
\qquad(c\longrightarrow+\infty).
\]
Consequently, for every $d\ge5$, a smooth lattice $d$-polytope with exactly $d+2$ facets can have a negative coefficient of $t^{d-4}$.
\end{theorem}

\begin{proof}
The product expression
\[
 \binom{t+cj+r}{r}=\frac1{r!}\prod_{i=1}^{r}(t+cj+i)
\]
shows that the coefficient of $c^r$ in the summand is $j^r/r!$. Every other term has $c$-degree at most $r-1$. After summation over $j$, these terms remain polynomials in $t$ by \eqref{eq:faulhaber} and the separate identity $S_0(t)=t+1$. Hence \eqref{eq:ranktwo-count} has the form
\[
 L_{Q_{r,c}}(t)=\frac{c^r}{r!}S_r(t)
 +\sum_{k=0}^{r-1}c^k A_{r,k}(t),
\]
where the polynomials $A_{r,k}(t)$ have rational coefficients independent of $c$.

The coefficient in \eqref{eq:bernoulli-coefficient} now gives
\[
 \begin{aligned}
 \coeff{t}{r-3}L_{Q_{r,c}}(t)
 &=\frac{c^r}{r!}
   \left(-\frac{r(r-1)(r-2)}{720}\right)
   +\sum_{k=0}^{r-1}c^k\coeff{t}{r-3}A_{r,k}(t)\\
 &=-\frac{c^r}{720(r-3)!}+O(c^{r-1}).
 \end{aligned}
\]
The coefficient is a polynomial in $c$ with negative leading coefficient. Its quotient by $c^r$ tends to $-1/(720(r-3)!)$, so it is negative for every sufficiently large integer $c$. \cref{prop:ranktwo-geometry} gives dimension $d=r+1$ and facet number $r+3=d+2$. The relevant exponent is $r-3=d-4\ge1$, as required.
\end{proof}

For $r=4$, an explicit five-dimensional example given 
\begin{align*}
 L_{Q_{4,c}}(t)&=\frac{c^4+5c^3+10c^2+10c+5}{120}t^5+\frac{c^4+9c^3+26c^2+34c+22}{48}t^4
 \\ & \quad +\frac{c^4+18c^3+83c^2+150c+135}{72}t^3 +\frac{5(c^3+9c^2+24c+34)}{48}t^2
 \\ &\quad+\frac{-c^4+175c^2+750c+2220}{720}t+1.
\end{align*}
In particular, at $c=16$ the polytope
\[
 Q_{4,16}=\{(x,z)\in\R^4\times\R:
 0\le z\le1,\ x_i\ge0\ (1\le i\le4),\
 x_1+x_2+x_3+x_4\le1+16z\}
\]
has seven facets and ten vertices, and its Ehrhart polynomial is
\[
L_{Q_{4,16}}(t)=\frac{88741}{120}t^5+\frac{54811}{24}t^4+\frac{54349}{24}t^3+\frac{17045}{24}t^2-\frac{181}{20}t+1.
\]

\begin{proof}[Proof of \cref{thm:minimum}]
Every full-dimensional $d$-polytope has at least $d+1$ facets; this standard fact is discussed in \cite[Chapter~2]{Ziegler1995}. In dimensions one and two, the leading, second highest, and constant Ehrhart coefficients account for all coefficients, and they are positive by the basic Ehrhart coefficient formulas. Thus $m_1=m_2=+\infty$.

In dimensions three and four, \cref{prop:positive} gives positivity whenever the number of facets is $d+1$ or $d+2$. Thus $m_d\ge d+3$ for $d=3,4$. The polytope $R_{1,9}$ has six facets and negative linear coefficient by \cref{prop:three-dimensional}; hence $m_3=6$. The polytope $R_{2,12}$ has seven facets and the negative coefficients displayed in \eqref{eq:four-dimensional-example}; hence $m_4=7$.

For $d\ge5$, \cref{prop:positive} gives positivity with $d+1$ facets, so $m_d\ge d+2$. \cref{thm:ranktwo-negative}, with $r=d-1$, supplies a smooth $d$-polytope with exactly $d+2$ facets and a negative Ehrhart coefficient. Thus $m_d=d+2$ for every $d\ge5$.
\end{proof}

\section{Higher integrality and smooth vertices}\label{sec:integrality}

We use higher integrality in the sense of Liu. Its definition involves an ordered lattice basis, and every face condition must hold with respect to that same basis. The following is the full-dimensional specialization of \cite[Definition~1.1]{Liu2011}.

\begin{definition}
An ordered basis $b_1,\ldots,b_d$ of $\Z^d$ determines coordinate projections $\pi_j:\R^d\longrightarrow\R^j$ by
\[
 \pi_j\left(\sum_{i=1}^{d}x_i b_i\right)=(x_1,\ldots,x_j),
 \qquad 0\le j\le d.
\]
Here $\R^0=\Z^0=\{0\}$. An affine subspace $A$ of dimension $j$ is \emph{integral with respect to this basis} if
\[
 \pi_j(A\cap\Z^d)=\Z^j.
\]
For an integer $0\le k\le d$, a lattice polytope $P$ is \emph{$k$-integral with respect to this basis} if the affine hull of every $j$-dimensional face is integral for every $0\le j\le k$.
\end{definition}

The condition for $j=0$ says that every vertex is a lattice point. Thus $0$-integrality is ordinary lattice integrality, and $k$-integrality implies $1$-integrality whenever $k\ge1$. The existence of an ordered basis satisfying the definition is invariant under affine unimodular transformations: such a transformation transports the lattice basis and the affine lattices of all faces.

We recall only the coefficient consequence of Liu's theorem that is needed here. For a polytope in $\R^j$, the notation $\Vol_j$ denotes volume measured so that a fundamental parallelepiped of $\Z^j$ has volume one. In particular, a unimodular $j$-simplex has volume $1/j!$, and the volume of the point in $\R^0$ is one.

\begin{theorem}[Liu; {\cite[Theorem~1.3]{Liu2011}}]
If a full-dimensional lattice $d$-polytope $P$ is $k$-integral with respect to an ordered lattice basis, then
\[
 e_j(P)=\Vol_j(\pi_j(P)),\qquad 0\le j\le k.
\]
\end{theorem}

For $d\ge2$, this formula implies that $(d-2)$-integrality is sufficient for Ehrhart positivity, as observed in \cite{Liu2011}. Indeed, all the projected polytopes have their indicated dimensions and positive volumes, and the leading two Ehrhart coefficients are positive for every lattice polytope. In the particular case $k=1$, the linear Ehrhart coefficient is the length of a lattice interval and is therefore a positive integer. The results below give elementary obstructions that do not require computation of an Ehrhart polynomial.

\begin{lemma}
A lattice polytope $P$ is $1$-integral with respect to an ordered lattice basis having first coordinate functional $\ell$ if and only if
\[|\ell(u)|=1
\]
for every primitive edge direction $u$ of $P$. Consequently, $P$ is $1$-integral with respect to some ordered lattice basis if and only if an integer linear functional satisfying these equations exists.
\end{lemma}

\begin{proof}
The definition of $1$-integrality imposes conditions on the vertices and
on the edges. Each zero-dimensional face is a singleton $\{v\}$, where
$v\in\Z^d$ because $P$ is a lattice polytope. Its affine hull is again
$\{v\}$, so
\[
 \aff(\{v\})\cap\Z^d=\{v\},
 \qquad
 \pi_0\bigl(\aff(\{v\})\cap\Z^d\bigr)=\{0\}=\Z^0.
\]
Thus every vertex satisfies the condition required in dimension zero.

Let $E$ be an edge with lattice endpoint $v$ and primitive direction $u$.
Its affine hull is the line $v+\R u$. We claim that the lattice points on
this line are exactly $v+\Z u$. Write $u=(u_1,\ldots,u_d)$ in standard
lattice coordinates. Since $u$ is primitive, the greatest common divisor
of its coordinates is one. Hence there are integers $a_1,\ldots,a_d$ with
$\sum_{i=1}^{d}a_i u_i=1$. If $v+\alpha u\in\Z^d$ for a real number
$\alpha$, then $\alpha u\in\Z^d$ and therefore
\[
 \alpha=\sum_{i=1}^{d}a_i(\alpha u_i)\in\Z.
\]
Conversely, $v+n u$ is a lattice point on the line for every $n\in\Z$.
Consequently,
\[
 \aff(E)\cap\Z^d=v+\Z u,
 \qquad
 \pi_1\bigl(\aff(E)\cap\Z^d\bigr)
 =\ell(v)+\ell(u)\Z.
\]
Both $\ell(v)$ and $\ell(u)$ are integers. The last set equals $\Z$ if
and only if $\ell(u)\Z=\Z$, or equivalently $|\ell(u)|=1$. The
conditions for vertices and edges together prove the first assertion.

The first coordinate functional of any ordered lattice basis is an
integer linear functional, so the first assertion gives one direction
of the final equivalence. Conversely, suppose that an integer linear
functional $\ell$ satisfies $|\ell(u)|=1$ for every primitive edge
direction. Choose one such direction and change its sign if necessary
to obtain $w\in\Z^d$ with $\ell(w)=1$. Every $x\in\Z^d$ then has the
decomposition
\[
 x=\ell(x)w+\bigl(x-\ell(x)w\bigr),
 \qquad x-\ell(x)w\in\ker\ell\cap\Z^d.
\]
This decomposition is unique: if $n w$ belongs to $\ker\ell$, then
$n=\ell(nw)=0$. Thus
\[
 \Z^d=\Z w\oplus(\ker\ell\cap\Z^d).
\]
The second summand has a lattice basis $b_2,\ldots,b_d$, and
$w,b_2,\ldots,b_d$ is an ordered basis of $\Z^d$. In this basis,
$\ell(w)=1$ and $\ell(b_i)=0$ for $2\le i\le d$, so $\ell$ is the first
coordinate functional. The first assertion now shows that $P$ is
$1$-integral with respect to this basis.
\end{proof}

\begin{lemma}
If $v$ is a smooth vertex of $P$, with primitive incident edge directions $u_1,\ldots,u_d$, then for every pair of distinct indices $i,j$ there is a two-dimensional face $F$ containing those two edges, with
\[
 \aff(F)=v+\spanR\{u_i,u_j\},
 \qquad
 \aff(F)\cap\Z^d=v+\Z u_i+\Z u_j.
\]
\end{lemma}

\begin{proof}
Orient each $u_k$ from $v$ into its incident edge. Since $v$ is smooth,
the vectors $u_1,\ldots,u_d$ form a basis of $\Z^d$. Let
$\lambda_1,\ldots,\lambda_d$ be the corresponding coordinate
functionals; explicitly,
\[
 \lambda_k(u_l)=
 \begin{cases}
  1,&k=l,\\
  0,&k\ne l,
 \end{cases}
 \qquad 1\le k,l\le d.
\]
Thus every $z\in\R^d$ has the unique expression
$z=\sum_{k=1}^{d}\lambda_k(z)u_k$, and each $\lambda_k$ takes integer
values on $\Z^d$.

The cone consisting of the nonnegative multiples of vectors in $P-v$
is generated by the directions of the edges incident to $v$; see
\cite{Ziegler1995}. Hence
\[
 \{\alpha(x-v):\alpha\ge0,\ x\in P\}
 =\left\{\sum_{k=1}^{d}s_k u_k:s_k\ge0\ (1\le k\le d)\right\}.
\]
In particular, if $x\in P$, then every coordinate
$\lambda_k(x-v)$ is nonnegative. Fix distinct indices $i,j$, and define
the linear functional
\[
 f(z)=\sum_{\substack{1\le k\le d\\k\notin\{i,j\}}}\lambda_k(z).
\]
We have $f(x-v)\ge0$ for all $x\in P$ and $f(v-v)=0$. Therefore the
points at which this functional attains its minimum form a face
\[
 F=\{x\in P:f(x-v)=0\}.
\]
For $x\in P$, all summands in $f(x-v)$ are nonnegative. Their sum is
zero precisely when
$\lambda_k(x-v)=0$ for every $k\notin\{i,j\}$. It follows that
\[
 F\subseteq v+\spanR\{u_i,u_j\}.
\]
Each point on the edge from $v$ in direction $u_i$ has the form
$v+s u_i$ with $s\ge0$ and satisfies $f(su_i)=0$. The same argument
applies to the edge in direction $u_j$. Thus both edges lie in $F$.
Their directions are linearly independent, so $F$ has dimension at
least two. The displayed containment gives the reverse bound and
proves
\[
 \dim F=2,
 \qquad \aff(F)=v+\spanR\{u_i,u_j\}.
\]
When $d=2$, the sum defining $f$ is empty and equals zero; in that case
$F=P$, and the same conclusion holds.

Let $z\in\aff(F)\cap\Z^d$. Since $v\in\Z^d$ and
$u_1,\ldots,u_d$ form a lattice basis, there are unique integers
$n_1,\ldots,n_d$ such that
\[
 z-v=\sum_{k=1}^{d}n_k u_k.
\]
On the other hand, the formula for $\aff(F)$ gives real numbers
$s,t$ with $z-v=s u_i+t u_j$. Uniqueness of coordinates in the real
basis $u_1,\ldots,u_d$ implies that $n_k=0$ for
$k\notin\{i,j\}$, $s=n_i$, and $t=n_j$. Therefore
$z\in v+\Z u_i+\Z u_j$. Conversely, for any integers $m,n$, the point
$v+m u_i+n u_j$ is a lattice point in
$v+\spanR\{u_i,u_j\}=\aff(F)$. These two inclusions prove
\[
 \aff(F)\cap\Z^d=v+\Z u_i+\Z u_j.
\]
This completes the proof.
\end{proof}

\begin{proof}[Proof of \cref{thm:obstruction}]
A contradiction follows if an ordered lattice basis makes $P$ $2$-integral. Its first two coordinate functionals will be denoted by $\ell_1,\ell_2$. The primitive incident directions at a smooth vertex $v$ will be denoted by $u_1,\ldots,u_d$. The edge criterion gives $\ell_1(u_i)=\pm1$ for every $i$.

For any distinct $i,j$, the preceding lemma supplies a two-face whose affine lattice is $v+\Z u_i+\Z u_j$. Its projection onto the first two coordinates must be all of $\Z^2$. Since the projection of $v$ is an integer vector, the two projected directions must generate $\Z^2$. The integer matrix with those directions as columns therefore has determinant $1$ or $-1$:
\[
 \ell_1(u_i)\ell_2(u_j)-\ell_1(u_j)\ell_2(u_i)=\pm1.
\]
Reduction modulo two, using $\ell_1(u_i)\equiv1\pmod2$, shows that $\ell_2(u_i)$ and $\ell_2(u_j)$ have opposite parity. This must hold for every pair. Among the three integers $\ell_2(u_1),\ell_2(u_2),\ell_2(u_3)$, however, two have the same parity. The contradiction proves the first assertion for every choice of ordered basis.

When $d\ge4$, the inequality $d-2\ge2$ shows that $(d-2)$-integrality would imply $2$-integrality. Every vertex of a smooth polytope is smooth, so the first assertion proves the remaining conclusion.
\end{proof}

The three-dimensional polytope $R_{1,9}$ fails even $1$-integrality. Three of its edges are
\[
 [(0,0,0),(1,0,0)],\quad
 [(0,0,0),(0,0,1)],\quad
 [(1,0,0),(10,0,1)].
\]
They are obtained by imposing, respectively, the facet equalities $y=z=0$, $x=y=0$, and $y=0$, $x-9z=1$. Their primitive directions are $e_1$, $e_3$, and $9e_1+e_3$. A functional satisfying the edge criterion would have $\ell(e_1),\ell(e_3)\in\{1,-1\}$ and therefore
\[
 \ell(9e_1+e_3)\in\{-10,-8,8,10\},
\]
contrary to that criterion. Together with its negative linear coefficient, this gives a six-facet example failing both properties in dimension three. The higher-dimensional counterexamples of \cref{thm:minimum} also fail the sufficient higher integrality condition by \cref{thm:obstruction}.

Failure of $1$-integrality already occurs for the standard smooth simplex $\Delta_d$ when $d\ge2$. Its primitive edge directions include $e_1$, $e_2$, and $e_2-e_1$. If the first two functional values are each $1$ or $-1$, their difference belongs to $\{-2,0,2\}$, which excludes the required value on the third direction. This simplex is Ehrhart positive and has $d+1$ facets, so even the smallest possible facet number does not force $1$-integrality.

The order of the obstruction is sharp. For the standard cube $[0,1]^d$, every primitive edge direction is a coordinate vector or its negative. The integer functional $\ell(x)=\sum_{i=1}^{d}x_i$ satisfies the edge criterion, and the cube is consequently $1$-integral with respect to a suitable ordered lattice basis. In dimensions at least three, \cref{thm:obstruction} prevents it from being $2$-integral in any basis.






\vspace{5mm}
\noindent
{\small \textbf{Acknowledgments:}}
Feihu Liu was partially supported by the Postdoctoral Fellowship Program and China Postdoctoral Science Foundation (Grant No. BX2026002).


\end{document}